\documentclass{amsart}

\usepackage{graphicx,color}

\makeatletter
\@addtoreset{equation}{section}

\makeatother

\newtheorem{theorem}{Theorem}[section]

\newtheorem{lemma}[theorem]{Lemma}
\theoremstyle{definition}
\newtheorem{definition}[theorem]{Definition}
\newtheorem{remark}[theorem]{Remark}

\begin{document}

\title[The isoperimetric inequality and surface diffusion flow]{On the isoperimetric inequality and surface diffusion flow for multiply winding curves}
\author[T.~Miura]{Tatsuya Miura}
\address[T.~Miura]{Department of Mathematics, Tokyo Institute of Technology, Meguro, Tokyo 152-8511, Japan}
\email{miura@math.titech.ac.jp}
\author[S.~Okabe]{Shinya Okabe}
\address[S.~Okabe]{Mathematical Institute, Tohoku University, Aoba, Sendai 980-8578, Japan}
\email{shinya.okabe@tohoku.ac.jp}

\date{\today}

\begin{abstract}
In this paper we establish a general form of the isoperimetric inequality for immersed closed curves (possibly non-convex) in the plane under rotational symmetry.
As an application we obtain a global existence result for the surface diffusion flow,
providing that an initial curve is $H^2$-close to a multiply covered circle and sufficiently rotationally symmetric.
\end{abstract}

\maketitle

\section{Introduction}\label{section:1}

It is well known that the behavior of the isoperimetric ratio plays an important role in the surface diffusion flow, which is a kind of higher order geometric flow.
In this paper we first establish a general form of the isoperimetric inequality for rotationally symmetric immersed closed curves in the plane, which are possibly non-convex, and then apply it to obtain a global existence result for the surface diffusion flow for curves, which we call the curve diffusion flow (CDF) for short.

\subsection{Isoperimetric inequality}

For a planar closed Lipschitz curve $\gamma$, let $\mathcal{L}(\gamma)$ and $\mathcal{A}(\gamma)$ denote the length and the signed area, respectively,
where we choose the area of a counterclockwise circle to be positive (see Section \ref{section:2} for details).
We define the isoperimetric ratio of $\gamma$ as
\begin{align}\label{eq:1.1}
I(\gamma) :=
\begin{cases}
\dfrac{\mathcal{L}(\gamma)^2}{4 \pi \mathcal{A}(\gamma)} & (\mathcal{A}(\gamma)>0),\\
\infty & (\mathcal{A}(\gamma)\leq0).
\end{cases}
\end{align}
The classical isoperimetric inequality asserts that $\inf I(\gamma) = 1$ in a certain class, and the infimum is attained if and only if $\gamma$ is a round circle, cf.\ \cite{Osserman_1978}.

Our first purpose is to obtain a generalized isoperimetric inequality that extracts the information of rotation number;
namely, we try to find a class $X_n$ of immersed closed curves such that $\inf_{X_n} I(\gamma) = n$, where $n\geq2$, so that the infimum is attained by an $n$-times covered circle.
This is however not easily done by the very simple idea to restrict admissible curves into $n$-times rotating curves.
Indeed, even in such a class the isoperimetric ratio can be arbitrarily close to $1$ due to an example of a large circle with small $(n-1)$-loops;
this example leads us to seek an appropriate ``global'' assumption on the admissible class.

In this paper we focus on rotationally symmetric curves.
For an integer $n\in\mathbb{Z}$ and a positive integer $m\in\mathbb{Z}_{>0}$,
we define the class $A_{n,m}$ to consist of all immersed curves in $W^{2,1}(\mathbb{S}^1;\mathbb{R}^2)$ of rotation number $n$ and of $m$-th rotational symmetry,
where we choose the counterclockwise rotation to be positive (see Definitions \ref{theorem:2.1} and \ref{theorem:2.2} for details).

We are now in a position to state our first main theorem, which gives a fully general version of the isoperimetric inequality for rotationally symmetric curves.

\begin{theorem}\label{theorem:1.1}
Let $n\in\mathbb{Z}$ and $m\in\mathbb{Z}_{>0}$.
Then
\begin{align}\label{eq:1.2}
\inf_{\gamma\in A_{n,m}}I(\gamma) = i_{n,m} := n+m-m\left\lceil\frac{n}{m}\right\rceil.
\end{align}
The index $i_{n,m}$ is nothing but a unique element in $(n+m\mathbb{Z})\cap\{1,\dots,m\}$, and $i_{n,m}=n$ holds if and only if $1 \leq n \leq m$.
The infimum in \eqref{eq:1.2} is attained if and only if $i_{n,m}=n$ and $\gamma$ is a counterclockwise $n$-times covered round circle.
\end{theorem}

Theorem \ref{theorem:1.1} covers general $n$ and $m$, although in our application we only use the case that $1\leq n\leq m$,
in which the lower bound $i_{n,m}$ is exactly the rotation number $n$.
However, our general statement would be of independent interest, and indeed highlights a difficulty to find appropriate global assumptions;
the infimum is not attained when $n\not\in[1,m]$ because a minimizing sequence may have small loops in symmetric positions and thus the limit curve may have a different rotation number,
which turns out to be $i_{n,m}$.
As an additional remark we mention that if $m=1$, then $i_{n,m}=1$ for any $n\in\mathbb{Z}$;
this means that the classical isoperimetric inequality in the class of immersed closed curves is completely retrieved.

Theorem \ref{theorem:1.1} is previously obtained in the subclass of $A_{n,m}$ that consists of locally convex curves.
To the best of our knowledge, this convex version is first shown by Epstein-Gage \cite[(5.11)]{EG_1987} for highly symmetric curves such that $1 \leq n <m/2$;
their result gives a Bonnensen-style sharper estimate.
The case that $1 \leq n \leq m$ is first proved by Chou \cite[Lemma 3.1]{Chou_2003}, and then by S\"{u}ssmann \cite[Theorem 3]{Suessmann_2011}
and by Wang-Li-Chao \cite[Theorem 4]{WLC_2017} in different ways.
However, all of them heavily rely on convexity in the sense that they use the parametrization by tangent angle.
The main novelty of our result is removing the convexity assumption; this point is crucial in our application.

In the proof of Theorem \ref{theorem:1.1}, we carry out a direct method for vector-valued functions themselves.
However, we then encounter an issue that the isoperimetric ratio only controls up to first order derivatives, while the rotation number is of second order;
this issue is directly related to ``vanishing loops'' phenomena.
Our key idea is to reduce the original problem for closed curves into a free boundary problem for open curves, by using symmetry.
By this procedure the rotation number is translated into the free boundary condition, and thereby converted into the index $i_{n,m}$.
Since the free boundary condition turns out to be of lower order, we are then able to employ a direct method in a class of open Lipschitz curves (Theorem \ref{theorem:3.1}).
Our free boundary problem might first look similar to a relative isometric problem in a sector, but be rather close to a problem of which ambient space is conical; see Remark~\ref{theorem:3.2} for more details.

We conclude this subsection by mentioning Banchoff and Pohl's isoperimetric inequality \cite{BP_1971}, which asserts without any symmetry that
\begin{equation}\label{eq:1.3}
\mathcal{L}(\gamma)^2 \geq 4\pi \int_{\mathbb{R}^2\setminus\gamma(\mathbb{S}^1)}w^2,
\end{equation}
where $w:\mathbb{R}^2\setminus\gamma(\mathbb{S}^1)\to\mathbb{Z}$ denotes the winding number with respect to $\gamma$,
and the equality holds if and only if $\gamma$ is a circle possibly multiply covered.
Their result possesses the strong advantage of being unified, but would not be compatible with our purpose (see Remark \ref{theorem:4.8}).

\subsection{Curve diffusion flow}

We apply Theorem \ref{theorem:1.1} to the Cauchy problem of the curve diffusion flow:
\begin{align}
\label{eq:CDF} \tag{CDF}
\begin{cases}
\partial_t \gamma = - (\partial^2_s \kappa) \nu, \\
\gamma(\cdot,0)=\gamma_0(\cdot),
\end{cases}
\end{align}
where $\gamma:\mathbb{S}^1\times[0,T)\to\mathbb{R}^2$ is a one-parameter family of immersed curves, and $\kappa$, $\nu$, and $s$ denote the signed curvature, the unit normal vector,
and the arc length parameter of each time-slice curve $\gamma(t):=\gamma(\cdot,t)$, respectively.
The curve diffusion flow decreases the length while preserving the (signed) area, and thus the isoperimetric ratio plays an important role.

The surface diffusion flow is first introduced by Mullins (\cite{Mullins_1957}) in 1957 and then studied by many authors.
The local well-posedness of \eqref{eq:CDF} is by now well known even in higher dimensions, thanks to its parabolicity;
see e.g. \cite{EG_1997,GI_1998,EMS_1998,DKS_2002} (and also more recent \cite{Asai_2010,EM_2010,KL_2012}).
On the other hand, the global behavior is more complicated even for curves, due to being of higher order;
the flow may lose some properties as e.g.\ embeddedness \cite{GI_1998}, convexity \cite{GI_1999}, and being a graph \cite{EMP_2001} (see also \cite{Chou_2003,EI_2005,Blatt_2010}).

Our goal is to capture certain initial curves that allow global-in-time solutions to \eqref{eq:CDF}.
It is known that even from a smoothly immersed initial curve the solution may develop a singularity in finite time \cite{Polden_1996, Chou_2003,EI_2005}, and in this case the total squared curvature always blows up \cite{DKS_2002,Chou_2003}.
However, if an initial state is suitably close to a round circle, then the solution exists globally-in-time and converges to a round circle (even in higher dimensions)
\cite{EG_1997,EMS_1998,Chou_2003,EM_2010,Wheeler_2012,Wheeler_2013}; this means that a round circle is a dynamically stable stationary solution.
In this paper we consider solutions that converge to multiply covered circles;
this seems to be a natural direction since all the stationary solutions in \eqref{eq:CDF} of immersed closed curves are circles possibly multiply covered.

Our main result asserts that if an initial curve is $H^2$-close to an $n$-times covered circle and also $m$-th rotationally symmetric for $1\leq n\leq m$,
then there exists a unique global-in-time solution.
Our proof combines Theorem \ref{theorem:1.1} with Wheeler's variational argument for the singly winding case \cite{Wheeler_2013}.
Hence, we use as a key quantity the normalized oscillation of curvature for a closed planar curve $\gamma$:
\begin{align}\label{eq:1.4}
K_{\rm osc}(\gamma) := \mathcal{L}(\gamma) \int_\gamma (\kappa - \bar{\kappa})^2 \, ds,
\quad \bar{\kappa} := \frac{1}{\mathcal{L}(\gamma)}\int_\gamma \kappa ds,
\end{align}
and also introduce an explicit constant related to the smallness condition:
\begin{align}\label{eq:1.5}
K^*_n := \dfrac{2\pi}{3} \left( \sqrt{1+3 n^2 \pi} - \sqrt{3 n^2 \pi} \right)^2
> 0,
\end{align}
which is same as the constant $K^*$ given in \cite[Proposition 3.6]{Wheeler_2013}.
Then we have:

\begin{theorem}\label{theorem:1.2}
Let $\gamma_0$ be a smoothly immersed initial curve.
Suppose that $\gamma_0 \in A_{n,m}$ for some $1\leq n\leq m$, and moreover there is some $K\in(0,K^*_n]$ such that
\begin{equation}
\label{eq:1.6}
K_{\rm osc}(\gamma_0) \leq K, \qquad \frac{I(\gamma_0)}{n} \leq \exp\left(\frac{K}{8 n^2 \pi^2}\right).
\end{equation}
Then \eqref{eq:CDF} admits a unique global-in-time solution $\gamma:\mathbb{S}^1\times[0,\infty)\to\mathbb{R}^2$.
In addition, the solution $\gamma$ satisfies
\begin{equation}\label{eq:1.7}
\sup_{t\in[0,\infty)}K_{\rm osc}(\gamma(t)) \le 2 K,
\end{equation}
retains symmetry in the sense that $\gamma(t)\in A_{n,m}$ for any $t\in[0,\infty)$, and smoothly converges as $t\to\infty$ to an $n$-times covered round circle of the same area as $\gamma_0$.
\end{theorem}

We need the smoothness of $\gamma_0$ just for local well-posedness; for example, it is enough if $\gamma_0\in C^{2,\alpha}(\mathbb{S}^1)$ for our purpose \cite[Theorem 1.1]{EMS_1998}
(see Section \ref{section:4}).

The assumption of rotational symmetry may not be optimal but we certainly need some other assumption than closeness to a circle,
because of the main difference from the singly winding case: multiply covered circles are not dynamically stable.
For example, if we take an initial curve as a doubly covered circle perturbed like a lima\c{c}on, then the small inner loop may vanish in finite time;
see a numerical computation \cite[Figure~1]{EMS_1998} and also Chou's elegant analytic proof \cite[Proposition~B]{Chou_2003}.
The first example of a nontrivial multiply winding global solution is numerically obtained in \cite[Figure~2]{EMS_1998}, under rotational symmetry.
In addition, the fact is that a similar statement to Theorem \ref{theorem:1.1} is previously claimed by Chou \cite[Proposition~C]{Chou_2003} for locally convex initial curves.
However, unfortunately, there seems to be a technical gap: Chou's argument implicitly uses the unverified property that the convexity of an initial curve is retained up to the maximal existence time, since the key isoperimetric inequality (\cite[Lemma 3.1]{Chou_2003}) is only
established for locally convex curves.
This flaw is now easily fixed by using Theorem \ref{theorem:1.1}, which does not require convexity, so in this way we can amend Chou's argument.
Notwithstanding, there are some advantages to follow Wheeler's argument instead of Chou's one: it allows us to deal with non-convex initial data, and also have the explicit smallness conditions (\ref{eq:1.6}) and (\ref{eq:1.7}).
In this sense our result not only corrects but also generalizes \cite[Proposition~C]{Chou_2003}.
In addition, we are also able to have a sharp estimate on the total time that a solution is non-convex (see Remark \ref{theorem:4.7}).

We finally mention that, after Abresch and Langer's celebrated study \cite{AL_1986}, the behavior of multiply winding rotationally symmetric curves is investigated by several authors
for the curve shortening flow \cite{EW_1987,EG_1987,Au_2010,Suessmann_2011,Wang_2011,CW_2014} and other second order geometric flows \cite{WK_2014,WW_2014,CWY_2017,WLC_2017};
in particular, a kind of stability result for multiply covered circles is obtained by Wang \cite{Wang_2011}, in the same spirit of our result.
We remark that all these studies focus on locally convex curves, thus using the property of second order.

This paper is organized as follows.
In Section \ref{section:2} we prepare notation and terminology more rigorously.
Sections \ref{section:3} and \ref{section:4} are devoted to the proofs of Theorems \ref{theorem:1.1} and \ref{theorem:1.2}, respectively.

\subsection*{Acknowledgements}
The first author of this paper was supported by JSPS KAKENHI Grant Number 18H03670, and the second author by JSPS KAKENHI Grant Numbers 19H05599 and 16H03946.
Part of this work was done while the authors were visiting the Institute for Mathematics and its Applications, University of Wollongong.
The authors acknowledge the hospitality and are grateful to Professor Glen Wheeler for his kind invitation and encouragement.

\section{Preliminaries}\label{section:2}

Let $I:=(0,1)$ and $\bar{I}$ be the closure.
Let $R_\theta$ denote the counterclockwise rotation matrix in $\mathbb{R}^2$ through an angle $\theta$; for simplicity, let $R$ denote $R_{\pi/2}$.
For any Lipschitz curve $\gamma\in W^{1,\infty}(I;\mathbb{R}^2)=C^{0,1}(\bar{I};\mathbb{R}^2)$, we define the length $\mathcal{L}$ and the signed area $\mathcal{A}$ (centered at the origin) by
\begin{align*}
\mathcal{L}(\gamma) &:= \int_I|\partial_x\gamma|\, dx, \quad	\mathcal{A}(\gamma) := -\frac{1}{2}\int_I \gamma\cdot R\partial_x\gamma \, dx.
\end{align*}
In what follows we repeatedly use the fact that any Lipschitz curve (not necessarily regular) can be reparameterized by the arclength,
or more generally by a constant speed parameterization, while both $\mathcal{L}$ and $\mathcal{A}$ are not changed (see Section \ref{section:A}).
This fact allows us to use the following expressions in terms of the arclength parameter $s\in[0,\mathcal{L}(\gamma)]$ and the unit normal $\nu:=R\partial_s\gamma$:
\begin{align*}
\mathcal{L}(\gamma) =\int_\gamma ds, \quad \mathcal{A}(\gamma) =-\frac{1}{2}\int_\gamma \gamma\cdot\nu \, ds.
\end{align*}
In addition, if $\gamma\in W^{2,1}(I;\mathbb{R}^2)\subset C^1(\bar{I};\mathbb{R}^2)$ and $\gamma$ is regular (or immersed, i.e., $\min|\partial_x\gamma|>0$),
then we define the rotation number $\mathcal{N}$ by
\begin{align*}
\mathcal{N}(\gamma):=\frac{1}{2\pi}\int_\gamma \kappa \, ds,
\end{align*}
where $\kappa:=\partial_s^2\gamma\cdot\nu$ denotes the signed curvature.
The orientations of $\nu$ and $\kappa$ are taken so that for a counterclockwise circle both $\mathcal{A}$ and $\mathcal{N}$ are positive.

We now rigorously define rotational symmetry, and then introduce the class $A_{n,m}$.

\begin{definition}[Rotational symmetry]\label{theorem:2.1}
For $m\in\mathbb{Z}_{>0}$, we say that a closed curve $\gamma:\mathbb{S}^1\to\mathbb{R}^2$, where $\mathbb{S}^1:=\mathbb{R}/\mathbb{Z}$, is {\em $m$-th rotationally symmetric}
if there exists some $i\in\{1,\dots,m\}$ such that $\gamma(x+1/m)=R_{2\pi i/m}\gamma(x)$ holds for any $x\in\mathbb{S}^1$.
We also use the term {\em $(m,i)$-th rotational symmetry} when we explicitly use the index $i$.
\end{definition}

\begin{definition}[Class $A_{n,m}$]\label{theorem:2.2}
For $n\in\mathbb{Z}$ and $m\in\mathbb{Z}_{>0}$, let $A_{n,m}$ be the set of all regular curves $\gamma\in W^{2,1}(\mathbb{S}^1;\mathbb{R}^2)$ that satisfy $\mathcal{N}(\gamma)=n$
and possess $m$-th rotational symmetry in the sense of Definition~\ref{theorem:2.1}.
\end{definition}

Note that the $W^{2,1}$-regularity is used for defining the rotation number.

We now observe that the index $i_{n,m}\in(n+m\mathbb{Z})\cap\{1,\dots,m\}$ in Theorem~\ref{theorem:1.1} naturally arises from the definition of $A_{n,m}$.

\begin{lemma}\label{theorem:2.3}
	Any curve $\gamma\in A_{n,m}$ is $(m,i_{n,m})$-th rotationally symmetric.
\end{lemma}

\begin{proof}
	Denote the first one period of $\gamma$ by $\gamma|_m:=\gamma|_{[0,1/m]}$.
	On one hand, the additivity $\mathcal{N}(\gamma)=m\mathcal{N}(\gamma|_m)$ implies that $\mathcal{N}(\gamma|_m)=n/m$.
	On the other hand, if $\gamma$ is $(m,i)$-th rotationally symmetric for some $i\in\{1,\dots,m\}$, then $\partial_x\gamma(L/m)=R_{2\pi i/m}\partial_x\gamma(0)$, and hence $\mathcal{N}(\gamma|_m)\in i/m+\mathbb{Z}$ holds, i.e., $n/m\in i/m+\mathbb{Z}$.
	Therefore, we find that $i\in(n+m\mathbb{Z})\cap\{1,\dots,m\}$, completing the proof.
\end{proof}

\section{The isoperimetric inequality}\label{section:3}

In this section we prove Theorem \ref{theorem:1.1}.
As is mentioned in Section \ref{section:1}, a key step is to prove a type of isoperimetric inequality for open curves.
For $\theta\in[0,2\pi]$, letting $v_\theta:=(\cos\theta,\sin\theta)\in\mathbb{R}^2$, we define the half-line $\Lambda_\theta\subset\mathbb{R}^2$ by $\Lambda_\theta:=\{\lambda v_\theta\mid \lambda\geq0 \}$.
Let
\begin{equation}
\label{eq:3.1}
X_\theta:=\{\gamma\in W^{1,\infty}(I;\mathbb{R}^2)\mid \gamma(0)\in\Lambda_0,\ \gamma(1)\in\Lambda_\theta,\ |\gamma(0)|=|\gamma(1)|,\ \mathcal{L}(\gamma)>0\}.
\end{equation}

\begin{theorem}\label{theorem:3.1}
For any $\theta\in(0,2\pi]$ and $\gamma\in X_\theta$, the following isoperimetric inequality holds:
\begin{align}\label{eq:3.2}
\mathcal{L}(\gamma)^2\geq 2\theta\mathcal{A}(\gamma).
\end{align}
The equality is attained if and only if $\gamma$ is a counterclockwise circular arc of central angle $\theta$;
in particular, if $\theta\in(0,2\pi)$, then the arc needs to be centered at the origin.
\end{theorem}

\begin{proof}
Without loss of generality, we may only consider curves of positive area since, otherwise, the inequality (\ref{eq:3.2}) obviously holds and the equality is not attained;
in addition, we may assume that $\mathcal{A}(\gamma)=1$ thanks to scale invariance.

Let $X_\theta':=\{\gamma\in X_\theta\mid \mathcal{A}(\gamma)=1\}$.
We first show that there exists $\bar{\gamma}\in X_\theta'$ such that $\mathcal{L}(\bar{\gamma})=\inf_{\gamma\in X_\theta'}\mathcal{L}(\gamma)$ by a direct method,
and then prove that $\bar{\gamma}$ is a circular arc of central angle $\theta$ by using a multiplier method.

{\em Step 1: Existence of a minimizer}.
Let $\ell:=\inf_{\gamma\in X_\theta'}\mathcal{L}(\gamma)\in[0,\infty)$.
Take a minimizing sequence $\{\gamma_n\}\subset X_\theta'$ such that $\lim_{n\to\infty}\mathcal{L}(\gamma_n)=\ell$.
Up to arclength reparameterization and normalization, we may assume that $\gamma_n$ is of constant speed on $I$, i.e., $|\partial_x\gamma_n|\equiv\mathcal{L}(\gamma_n)$;
in particular, $\{\partial_x\gamma_n\}_n$ is bounded in $L^\infty(I;\mathbb{R}^2)$.
In addition, we may also assume the boundedness of $\{\gamma_n\}$ in $L^\infty(I;\mathbb{R}^2)$; indeed, if $\theta\in(0,2\pi)$,
then from the boundary condition and elementary geometry we deduce that there is $C_\theta>0$ (say, $C_\theta=\frac{1}{2\sin(\theta/2)}$) such that for any $\gamma\in X_\theta'$,
\begin{align*}
|\gamma(0)|(=|\gamma(1)|)=C_\theta|\gamma(1)-\gamma(0)|\leq C_\theta\mathcal{L}(\gamma),
\end{align*}
and hence we obtain the desired boundedness of $\{\gamma_n\}$ since
\begin{align*}
\|\gamma_n\|_\infty\leq |\gamma_n(0)|+\mathcal{L}(\gamma_n)\leq (C_\theta+1)\mathcal{L}(\gamma_n);
\end{align*}
if $\theta=2\pi$, up to translation we may assume that the endpoints of $\gamma_n$ are pinned at the origin, and hence we similarly obtain the desired boundedness.
Therefore, there is a subsequence (which we denote by the same notation) that converges weakly in $H^1$ and strongly in $L^\infty$ to some $\bar{\gamma}\in W^{1,\infty}(I;\mathbb{R}^2)$.
Since $\mathcal{A}$ is continuous with respect to these convergences, we have $\mathcal{A}(\bar{\gamma})=\lim_{n\to\infty}\mathcal{A}(\gamma_n)=1$;
this in particular implies that $\mathcal{L}(\bar{\gamma})>0$; hence, noting that $\bar{\gamma}$ still satisfies the boundary condition, we find that $\bar{\gamma}\in X_\theta'$.
By the lower semicontinuity of $\mathcal{L}$ with respect to e.g.\ $L^\infty$-convergence, we have $\mathcal{L}(\bar{\gamma})\leq\liminf_{n\to\infty}\mathcal{L}(\gamma_n)=\ell$
and hence $\mathcal{L}(\bar{\gamma})=\ell$ ($>0$); therefore, $\bar{\gamma}$ is nothing but a minimizer.

{\em Step2: Any minimizer is a circular arc of central angle $\theta$.}
Fix any minimizer $\bar{\gamma}$ of $\mathcal{L}$ in $X_\theta'$.
Using the Lagrange multiplier method and calculating the first variation for interior perturbations, we find that $\bar{\gamma}$ is smooth on $\bar{I}$
and has constant curvature (see Lemma \ref{theorem:B.2}), thus being a circular arc that is possibly multiply covered.

We first complete the proof in the case that $\theta=2\pi$.
In this case the boundary condition that $\bar{\gamma}(0)=\bar{\gamma}(1)$ implies that $\bar{\gamma}$ is a closed circular arc,
i.e., the central angle of $\bar{\gamma}$ is of the form $2\pi j$, where $j$ is a positive integer by area-positivity.
We now only need to ensure that $j=1$;
this easily follows since thanks to the constraint $\mathcal{A}(\bar{\gamma})=1$, a direct computation implies that $\mathcal{L}(\bar{\gamma})=2\sqrt{\pi j}$,
and hence $j$ needs to be $1$ by length-minimality of $\bar{\gamma}$.
Therefore, $\bar{\gamma}$ is nothing but a counterclockwise circle.

Hereafter we assume that $\theta\in(0,2\pi)$.
We first notice that
\begin{align}\label{eq:3.3}
|\bar{\gamma}(0)|=|\bar{\gamma}(1)|>0.
\end{align}
Indeed, otherwise, $\bar{\gamma}$ needs to be closed and hence the length needs to be at least $2\sqrt{\pi}$ by the above argument, but this contradicts the length-minimality of $\bar{\gamma}$
since there is a smaller-length competitor in $X_\theta'$, e.g.\ the circular arc of central angle $\theta$, whose length is $\sqrt{2\theta}\in(0,2\sqrt{\pi})$.
The condition (\ref{eq:3.3}) allows us to perturb $\bar{\gamma}$ at the endpoints in the directions of the half-lines $\Lambda_0$ and $\Lambda_\theta$.
Using the multiplier method again (see Lemma \ref{theorem:B.3}), we find that
\begin{equation}\label{eq:3.4}
	\partial_x\bar{\gamma}(0)\cdot v_0=\partial_x\bar{\gamma}(1)\cdot v_\theta.
\end{equation}
(Geometrically speaking, this means that by gluing the two half-lines, the endpoints of $\bar{\gamma}$ need to be smoothly connected.)
On the other hand, since the arc $\bar{\gamma}$ is circular and $|\gamma(0)|=|\gamma(1)|$, an elementary geometry implies that the center of the arc $\bar{\gamma}$ lies on the bisector $\Lambda_{\theta/2}\cup\Lambda_{\theta/2+\pi}$, and hence
\begin{equation}\label{eq:3.5}
	\partial_x\bar{\gamma}(0)\cdot v_0=-\partial_x\bar{\gamma}(1)\cdot v_\theta.
\end{equation}
Combining \eqref{eq:3.4} and \eqref{eq:3.5}, we find that $\partial_x\bar{\gamma}(0)\cdot v_0=\partial_x\bar{\gamma}(1)\cdot v_\theta=0$, i.e., the circular arc $\bar{\gamma}$ meets perpendicularly $\Lambda_0$ and $\Lambda_\theta$ at the endpoints, respectively.
Therefore, $\bar{\gamma}$ is in particular centered at the origin and the central angle is of the form $2\pi j+\theta$, where $j$ is a nonnegative integer by area-positivity.
It turns out that $j=0$ by a direct computation which is parallel to the case of $\theta=2\pi$.
The proof is now complete.
\end{proof}

Before completing the proof of Theorem \ref{theorem:1.1}, we give some remarks on the above free boundary problem by comparing it with previous studies.

\begin{remark} \label{theorem:3.2}
Many kinds of relative isoperimetric inequalities have been studied for manifolds-with-boundary (see e.g.\ a survey \cite{Ros}),
including singular boundaries of sectorial type \cite{Bandle_1972,Bandle_1980,Choe_1996,BH_1998,Bahn_1999,Bahn_2000,Gimenez_2007} (or more generally of conical type;
see \cite{LP,RR,FI_2013,BF_2017}, and also \cite[Section~5]{Cabre} and references therein).
Theorem \ref{theorem:3.1} looks like a kind of relative isoperimetric problem of sectorial type, but is essentially different in the sense that even a ``non-convex" circular sector ($\theta>\pi$) appears as a minimizer, due to our additional constraint that $|\gamma(0)|=|\gamma(1)|$.

In fact, if we eliminate this constraint, then a similar argument to the proof of Theorem \ref{theorem:3.1} implies that a minimizer for $\theta>\pi$ is always a hemicircle such that
one endpoint lies at the origin, cf. \cite[Lemma~1]{Choe_1996}.
(This kind of phenomenon is observed even in higher dimensions \cite{Kim_2000,Choe_2003,CGR}).
A key point is that the endpoint of a minimizer needs to satisfy the right angle condition unless it is at the origin, since any perturbation along the half-line is allowed.

On the other hand, under the constraint that $|\gamma(0)|=|\gamma(1)|$, the right-angle condition is not as trivial as above since a first variation argument only implies a weak free boundary condition,
which means that a minimizer is also smooth near the endpoints in the space made by gluing the two half-lines.
Combining this condition with global symmetry, which is now an elementary geometry, we retrieve the desired right-angle condition (as in Lemma \ref{theorem:B.3}).
From this point of view, our problem may be rather regarded as a variant of isoperimetric problems of which ambient spaces are cones (see e.g. \cite[Section~2.2]{BZ} or \cite{MR}).

As another difference, the previous studies address only embedded curves that are entirely contained in the sectorial region surrounded by the half-lines $\Lambda_0$ and $\Lambda_\theta$,
but our problem admit any self-intersections and also going out of the sectorial region.
In this sense our direct method for vector-valued functions allows us to generalize the results in \cite{Bandle_1972,Bandle_1980} and \cite[Lemma~1]{Choe_1996}.
\end{remark}

We now complete the proof of Theorem \ref{theorem:1.1} by using Theorem \ref{theorem:3.1}.

\begin{proof}[Proof of Theorem {\rm \ref{theorem:1.1}}]
Throughout the proof, given a curve $\gamma\in A_{n,m}$, we denote the one period of $\gamma$ by $\gamma|_m:=\gamma|_{[0,1/m]}$.

We first prove the following inequality for an arbitrary $\gamma\in A_{n,m}$:
\begin{align}\label{eq:3.6}
I(\gamma) \geq i_{n,m}.
\end{align}
By Lemma \ref{theorem:2.3}, any $\gamma\in A_{n,m}$ is $(m,i_{n,m})$-rotationally symmetric and hence, up to rotation,
we may assume that $\gamma|_m\in X_\theta$ for $\theta:=2\pi i_{n,m}/m\in(0,2\pi]$ without loss of generality.
Therefore, by Theorem \ref{theorem:3.1},
\begin{align}\label{eq:3.7}
2\left(2\pi i_{n,m}\over m\right)\mathcal{A}(\gamma|_m)\leq\mathcal{L}(\gamma|_m)^2.
\end{align}
Using the additivities $\mathcal{L}(\gamma)=m\mathcal{L}(\gamma|_m)$ and $\mathcal{A}(\gamma)=m\mathcal{A}(\gamma|_m)$ due to the $m$-rotational symmetry,
we obtain $4\pi i_{n,m}\mathcal{A}(\gamma)\leq \mathcal{L}(\gamma)^2$, which is equivalent to (\ref{eq:3.6}).

We now prove (\ref{theorem:1.1}).
If $1\leq n \leq m$, then $i_{n,m}=n$, and hence (\ref{theorem:1.1}) follows from the inequality (\ref{eq:3.6})
and the fact that a counterclockwise $n$-times covered circle belongs to $A_{n,m}$ and attains the equality in (\ref{eq:3.6}).
For $n$ and $m$ such that $n\not\in[1,m]$, it suffices to construct a sequence of curves $\{\gamma_j\}_j\subset A_{n,m}$ such that $I(\gamma_j)\to i_{n,m}$.
This is easily done by adding small loops to a fixed counterclockwise $i_{n,m}$-times covered circle, which we denote by $\bar{\gamma}$, so that each $\gamma_j$ has rotation number $n$,
i.e., $\{\gamma_j\}_j\subset A_{n,m}$, but the added loops vanish as $j\to\infty$.
More precisely, we define $\gamma_j$ in such a way that the one period $(\gamma_j)|_m$ is given by a curve $\bar{\gamma}|_m$ with additional counterclockwise $(1-\lceil n/m \rceil)$-loops
of radius $1/j$, where we interpret negative $1-\lceil n/m \rceil$ as adding clockwise $(\lceil n/m \rceil-1)$-loops.
Then the total number of loops added to the entire curve $\bar{\gamma}$ is $m(1-\lceil n/m \rceil)=n-i_{n,m}$.
Therefore, $\mathcal{N}(\gamma_j)=\mathcal{N}(\bar{\gamma})+(n-i_{n,m})=n$ and hence $\gamma_j\in A_{n,m}$.
As $j\to\infty$, since the loops vanish, the isoperimetric ratio of $\gamma_j$ converges to that of the $i_{n,m}$-times covered circle, i.e., $I(\gamma_j)\to i_{n,m}$.

The remaining part is to prove that the infimum in (\ref{eq:1.2}) is attained only if $1\leq n\leq m$ and $\gamma$ is a counterclockwise $n$-times covered circle.
Suppose that a curve $\gamma\in A_{n,m}$ attains the equality in (\ref{eq:3.6}); then, clearly $\gamma|_m\in X_\theta$ attains the equality in (\ref{eq:3.7}),
namely, $\gamma|_m\in X_\theta$ is a counterclockwise circular arc of central angle $2\pi i_{n,m}/m$, and moreover centered at the origin when $i_{n,m}<m$.
This implies that $\gamma$ is a counterclockwise $i_{n,m}$-times covered circle.
Since $\gamma$ is necessarily admissible, i.e., $\mathcal{N}(\gamma)=n$, we now deduce $i_{n,m}=n$, which is equivalent to $1\leq n\leq m$.
The proof is now complete.
\end{proof}

\section{Curve diffusion flow}\label{section:4}

Throughout this section, we use the term ``smooth initial curve'' in the sense that $\gamma_0:\mathbb{S}^1\to\mathbb{R}^2$ is immersed and sufficiently smooth
so that (\ref{eq:CDF}) is locally well-posed in the following sense:
There exists a unique one-parameter family of immersed curves $\gamma:\mathbb{S}^1\times[0,T)$ such that $\gamma$ is smooth in $\mathbb{S}^1\times(0,T)$ and satisfies (\ref{eq:CDF}) in the classical sense,
and moreover $\gamma(t):=\gamma(\cdot,t)$ converges to $\gamma_0$ as $t \downarrow 0$ in $H^2(\mathbb{S}^1)$ so that all the quantities $\mathcal{A}$, $\mathcal{L}$, $\mathcal{N}$,
and $K_{\rm osc}$ are continuous on $[0,T)$.
It is known that this well-posedness is valid at least if $\gamma_0\in C^{2,\alpha}(\mathbb{S}^1)$ \cite[Theorem 1.1]{EMS_1998}.
Since there is a maximal value of $T$ such that the above well-posedness holds, we let $T_M(\gamma_0)\in(0,\infty]$ denote the maximal existence time corresponding to $\gamma_0$.

It is classically known that under the well-posedness, $\mathcal{A}$ and $\mathcal{N}$ are invariant and $\mathcal{L}$ is non-increasing along the flow;
these facts are obtained by differentiating in $t\in(0,T_M(\gamma_0))$ and using continuity at $t=0$.
We only state the facts we will use (see e.g.\ \cite[Lemmas 3.1 and 3.4]{Wheeler_2013} for details).

\begin{lemma}\label{theorem:4.1}
Let $\gamma_0$ be a smooth initial curve, and $\gamma$ be a unique solution to \eqref{eq:CDF}.
Then $\mathcal{A}(\gamma(t))=\mathcal{A}(\gamma_0)$, $\mathcal{L}(\gamma(t))\leq\mathcal{L}(\gamma_0)$, and $\mathcal{N}(\gamma(t))=\mathcal{N}(\gamma_0)$ hold for all $t\in[0,T_M(\gamma_0))$.
\end{lemma}

We finally review a simple sufficient condition for the global existence in (\ref{eq:CDF}), which plays a crucial role in our argument:
If the total squared curvature is uniformly bounded, then a solution exists globally-in-time \cite[Theorem 3.1]{DKS_2002} (see also \cite{Chou_2003}), i.e.,
\begin{align}\label{eq:4.1}
\sup_{t\in[0,T_M(\gamma_0))}\int_{\gamma(t)}\kappa^2ds<\infty \quad \Longrightarrow \quad T_M(\gamma_0)=\infty.
\end{align}
This condition can be translated in terms of the oscillation $K_{\rm osc}$ and the length $\mathcal{L}$.

\begin{lemma}\label{theorem:4.2}
Let $\gamma_0$ be a smooth initial curve, and $\gamma$ be a unique solution to \eqref{eq:CDF} with a maximal existence time $T_M(\gamma_0)$.
If
\begin{align}\label{eq:4.2}
\inf_{t\in[0,T_M(\gamma_0))}\mathcal{L}(\gamma(t))>0, \quad \sup_{t\in[0,T_M(\gamma_0))}K_{\rm osc}(\gamma(t))<\infty,
\end{align}
then $T_M(\gamma_0)=\infty$.
\end{lemma}

\begin{proof}
By the definitions of $K_{\rm osc}$ and $\bar{\kappa}$ and the relation $\bar{\kappa}=2\pi\mathcal{N}/\mathcal{L}$ we have
\begin{align*}
K_{\rm osc}(\gamma(t))
&= \mathcal{L}(\gamma(t)) \int_{\gamma(t)} ( \kappa^2 - 2 \bar{\kappa} \kappa + \bar{\kappa}^2 ) \, ds \\
&= \mathcal{L}(\gamma(t)) \int_{\gamma(t)} \kappa^2 \, ds - \mathcal{L}(\gamma(t))^2 \bar{\kappa}^2
 = \mathcal{L}(\gamma(t)) \int_{\gamma(t)} \kappa^2 \, ds - 4 \pi^2 \mathcal{N}(\gamma(t))^2.
\end{align*}
Combining this identity with (\ref{eq:4.2}) and the fact that $\mathcal{N}(\gamma(t))=\mathcal{N}(\gamma_0)$ in Lemma \ref{theorem:4.1},
we find that $\int_{\gamma(t)}\kappa^2ds$ is uniformly bounded, and hence $T_M(\gamma_0)=\infty$ by (\ref{eq:4.1}).
\end{proof}


\subsection{Global existence}

From now on we are going to prove Theorem \ref{theorem:1.2} by using Lemma \ref{theorem:4.2}.
The fact is that the uniform positivity of length is a generic property; indeed, the classical isoperimetric inequality (Theorem \ref{theorem:1.1} for $m=1$)
and the area-preserving property (Lemma \ref{theorem:4.1}) imply that $\mathcal{L}(\gamma(t))^2 \geq 4\pi\mathcal{A}(\gamma(t)) = 4\pi\mathcal{A}(\gamma_0)$,
and hence $\inf_{t}\mathcal{L}(\gamma(t))>0$ at least if $\mathcal{A}(\gamma_0)>0$ (or more generally if $\mathcal{A}(\gamma_0)\neq0$, thanks to invariance under change of orientation).
Thus we are mainly concerned with the oscillation of curvature.

Our starting point is to use the following upper bound of $K_{\rm osc}$ obtained by Wheeler \cite{Wheeler_2013}.
Recall that $K_{\rm osc}$ and $K_n^*$ are defined in (\ref{eq:1.4}) and (\ref{eq:1.5}).

\begin{lemma}[Control of the oscillation of curvature {\cite[Proposition 3.6]{Wheeler_2013}}] \label{theorem:4.3}
Let $\gamma_0$ be a smooth initial curve, and $\gamma$ be a unique solution to \eqref{eq:CDF}.
If there exists $T^*\in(0,T_M(\gamma_0)]$ such that
\begin{equation}
\label{eq:4.3}
\sup_{t\in[0,T^*)}K_{\rm osc}(\gamma(t)) \le 2 K^*_n,
\end{equation}
then for any $t\in[0,T^*)$,
\begin{align}\label{eq:4.4}
K_{\rm osc}(\gamma(t)) \le K_{\rm osc}(\gamma_0) + 8 \pi^2 n^2 \log{\frac{\mathcal{L}(\gamma_0)}{\mathcal{L}(\gamma(t))}}.
\end{align}
\end{lemma}

\begin{remark}
The original estimate in \cite[Proposition 3.6]{Wheeler_2013} is the following form:
\begin{align*}
K_{\rm osc}(\gamma(t)) + 8 \pi^2 n^2 \log{\mathcal{L}(\gamma(t))} + \int^t_0 K_{\rm osc}(\gamma(\tau)) \dfrac{\| \partial_s \kappa \|^2_{L^2}}{\mathcal{L}(\gamma(\tau))} \, d\tau\\
\qquad  \le K_{\rm osc}(\gamma_0) + 8 \pi^2 n^2 \log{\mathcal{L}(\gamma_0)}.
\end{align*}
However, we easily obtain (\ref{eq:4.4}) from this estimate by deleting the non-negative third term of the l.h.s., and moving the second term of the l.h.s.\ to the r.h.s.
\end{remark}

From Lemma \ref{theorem:4.3} we deduce that once the ratio $\mathcal{L}(\gamma_0)/\mathcal{L}(\gamma(t))$ is well controlled to be uniformly close to $1$, then assuming $K_{\rm osc}(\gamma_0)\ll1$,
we get a uniform control of $K_{\rm osc}$.
(Recall that $\mathcal{L}(\gamma_0)/\mathcal{L}(\gamma(t)) \geq 1$ by Lemma \ref{theorem:4.1}.)
In \cite[Proposition 3.7]{Wheeler_2013} Wheeler indeed controls $K_{\rm osc}$ in such a way, focusing on the case that $n=1$, and hence $\gamma_0$ is taken to be close to a round circle;
his key idea is to use the isoperimetric inequality.

Now, as a main application of Theorem \ref{theorem:1.1}, we state the following key lemma which gives a uniform control of $\mathcal{L}(\gamma_0)/\mathcal{L}(\gamma(t))$ under symmetry.

\begin{lemma}[Lower bound of length]\label{theorem:4.5}
Let $\gamma_0$ be a smooth initial curve, and $\gamma$ be a unique solution to \eqref{eq:CDF}.
If $\gamma_0 \in A_{n,m}$, then $\gamma(t)\in A_{n,m}$ for any $t\in[0,T_M(\gamma_0))$.
Moreover, if we additionally assume that $1 \leq n \leq m$ and $\mathcal{A}(\gamma_0)>0$, then for any $t\in[0,T_M(\gamma_0))$,
\begin{equation}\label{eq:4.5}
\frac{\mathcal{L}(\gamma_0)}{\mathcal{L}(\gamma(t))} \leq \sqrt{\frac{I(\gamma_0)}{n}}.
\end{equation}
\end{lemma}

\begin{proof}
The assertion that $\gamma(t)\in A_{n,m}$ follows from the well-posedness of (\ref{eq:CDF});
indeed, the property that $\mathcal{N}(\gamma(t))=\mathcal{N}(\gamma_0)=n$ follows from Lemma~\ref{theorem:4.1},
while the $m$-th rotational symmetry is also preserved along the flow thanks to uniqueness and geometric invariance of the flow (see Lemma~\ref{theorem:C.3} for a complete proof).

If we additionally assume that $1 \leq n \leq m$ and $\mathcal{A}(\gamma_0)>0$,
then from Theorem~\ref{theorem:1.1} and the fact that $\gamma(t)\in A_{n,m}$ we deduce that $\mathcal{L}(\gamma(t)) \geq \sqrt{4 \pi n \mathcal{A}(\gamma(t))}$.
Using the area-preserving property in Lemma \ref{theorem:4.1}, and recalling (\ref{eq:1.1}), we obtain
\begin{align*}
\mathcal{L}(\gamma(t)) \geq \sqrt{4 \pi n \mathcal{A}(\gamma(t))} = \sqrt{4 \pi n \mathcal{A}(\gamma_0)} = \mathcal{L}(\gamma_0) \sqrt{\frac{n}{I(\gamma_0)}},
\end{align*}
which completes the proof.
\end{proof}

We are now in a position to assert the main global existence result.

\begin{theorem}[Global existence]\label{theorem:4.6}
Let $\gamma_0$ be a smooth initial curve, and $\gamma$ be a unique solution to \eqref{eq:CDF}.
Suppose that $\gamma_0 \in A_{n,m}$ for some $1\leq n\leq m$, and moreover there is some $K\in(0,K^*_n]$ such that \eqref{eq:1.6} holds.
Then
\begin{equation}\label{eq:4.6}
\sup_{t\in[0,T_M(\gamma_0))}K_{\rm osc}(\gamma(t)) \leq 2K.
\end{equation}
In particular, $T_M(\gamma_0)=\infty$.
\end{theorem}

\begin{proof}
We first note that $\mathcal{A}(\gamma_0)>0$ under (\ref{eq:1.6}) since $I(\gamma_0)<\infty$, cf.\ (\ref{eq:1.1}).
Since $K_{\rm osc}(\gamma(t))$ is continuous and $K_{\rm osc}(\gamma(0)) \leq K$ by (\ref{eq:1.6}), there exists $T_K\in(0,T_M(\gamma_0)]$ such that
$$
T_K := \sup\{ \tau \in [0, T_M(\gamma_0)) \mid K_{\rm osc}(\gamma(t)) \le 2 K \quad \text{for} \quad t \in [0, \tau)\}.
$$
We prove $T_K=T_M(\gamma_0)$ by contradiction.
Assume that $T_K < T_M(\gamma_0)$.
Using Lemmas \ref{theorem:4.3} and \ref{theorem:4.5}, we have for $t\in[0, T_K)$,
\begin{align*}
K_{\rm osc}(\gamma(t))
& \stackrel{\eqref{eq:4.4}}{\le} K_{\rm osc}(\gamma_0) + 8 \pi^2 n^2 \log{\dfrac{\mathcal{L}(\gamma_0)}{\mathcal{L}(\gamma(t))}}\\
& \stackrel{\eqref{eq:4.5}}{\le} K_{\rm osc}(\gamma_0) + 4 \pi^2 n^2 \log{\frac{I(\gamma_0)}{n}}\\
& \stackrel{\eqref{eq:1.6}}{\le} K + 4 \pi^2 n^2 \cdot \dfrac{K}{8 \pi^2 n^2} = \dfrac{3}{2} K.
\end{align*}
This means that $K_{\rm osc}$ remains less than $2K$ in $[0,T_K+\varepsilon)$ for some small $\varepsilon>0$, but this contradicts to the maximality of $T_K$.
Therefore, we have $T_K=T_M(\gamma_0)$.

The assertion that $T_M(\gamma_0)=\infty$ now immediately follows from Lemma \ref{theorem:4.2},
thanks to the bounds of $K_{\rm osc}$, cf.\ \eqref{eq:4.6}, and $\mathcal{L}$, e.g., \eqref{eq:4.5}.
\end{proof}

We finally complete the proof of Theorem \ref{theorem:1.2}.

\begin{proof}[Proof of Theorem \ref{theorem:1.2}]
All the assertions except for the asymptotic behavior as $t\to\infty$ are already verified in Theorem \ref{theorem:4.6} and Lemma \ref{theorem:4.5}.
Since the remaining convergence directly follows from \cite[Proposition A]{Chou_2003} (or \cite[Section 4]{Wheeler_2013}), we do not repeat here.
The proof is now complete.
\end{proof}

\begin{remark}[Waiting time]\label{theorem:4.7}
Using Theorem \ref{theorem:1.1} (or Lemma \ref{theorem:4.5}), we are also able to argue about the ``waiting time'' \`{a} la Wheeler \cite[Proposition 1.5]{Wheeler_2013}.
For a given smooth initial curve $\gamma_0$, we let $T_W(\gamma_0)$ be the one-dimensional Lebesgue measure of the subset of the time interval $[0,T_M(\gamma_0))$
in which the solution to \eqref{eq:CDF} starting from $\gamma_0$ is not strictly convex.
His argument in \cite[p.945]{Wheeler_2013} gives the following general upper bound of $T_W(\gamma_0)$:
$$
T_W(\gamma_0) \leq \frac{\mathcal{L}(\gamma_0)^4-(4 \pi \mathcal{A}(\gamma_0))^2}{16 \pi^2 \mathcal{N}(\gamma_0)^2}.
$$
This is sharp in the sense that the right-hand side is zero if and only if $\gamma_0$ is circle, but hence not sharp when we consider solutions near multiply covered circles,
for which $\mathcal{L}(\gamma_0)^4$ is close to $\mathcal{N}(\gamma_0)^2(4 \pi \mathcal{A}(\gamma_0))^2$.
However, if $\gamma_0\in A_{n,m}$, then just replacing the classical isoperimetric inequality in his argument with our one, we reach the desired estimate:
$$
T_W(\gamma_0) \leq \frac{\mathcal{L}(\gamma_0)^4-(4 \pi n \mathcal{A}(\gamma_0))^2}{16 \pi^2 n^2}.
$$
\end{remark}

\begin{remark}[Banchoff-Pohl's isoperimetric inequality]\label{theorem:4.8}
We finally indicate that the area functional in (\ref{eq:1.3}),
$$\widehat{\mathcal{A}}(\gamma):=\int_{\mathbb{R}^2\setminus\gamma(\mathbb{S}^1)}w^2,$$
is not generally preserved by \eqref{eq:CDF}.
A simple example is the shrinking figure-eight \cite{Polden_1996,EMS_1998}, for which $\widehat{\mathcal{A}}$ strictly decreases.
More non-trivially, this phenomenon occurs even for locally convex initial curves.
Consider an initial curve $\gamma_0$ that is smoothly close to a multiply covered circle but blows up in finite time; such a curve exists in view of \cite[Proposition B]{Chou_2003}.
If the area $\widehat{\mathcal{A}}$ were preserved during the flow from $\gamma_0$,
then the same argument as deriving \eqref{eq:4.5} would imply that $\mathcal{L}(\gamma_0)/\mathcal{L}(\gamma(t))\leq(\widehat{I} (\gamma_0))^{1/2}$,
where $\widehat{I}:=\mathcal{L}/4\pi\widehat{\mathcal{A}}$, and since $\log\widehat{I}(\gamma_0) \ll 1$,
as in the proof of Theorem \ref{theorem:4.6} we would obtain a global solution $\gamma(t)$; this is a contradiction.
\end{remark}

\begin{appendix}

\section{Change of variables}\label{section:A}

In this section we verify that we can always use change of variables for Lipschitz curves, without any condition concerning the speed of a curve.
This point is important in our direct method since for a minimizing sequence of the isoperimetric ratio, the speed of a curve may degenerate in the limit.

Fix any $\gamma\in W^{1,\infty}(I;\mathbb{R}^2)$.
Let $J:=(0,\mathcal{L}(\gamma))$.
Then the function $\sigma:\bar{I}\to\bar{J}$ defined by
\begin{align*}
\sigma(x):=\int_0^x|\partial_x\gamma|
\end{align*}
is monotone, Lipschitz continuous, and satisfies $\partial_x\sigma=|\partial_x\gamma|$ a.e.\ in $I$.
In addition, there is a unique $1$-Lipschitz curve $\tilde{\gamma}\in W^{1,\infty}(J;\mathbb{R}^2)$ (corresponding to the arclength parameterization of $\gamma$) such that
\begin{equation}\label{eq:A.1}
\gamma(x)=\widetilde{\gamma}\circ\sigma(x) \qquad \text{for a.e.}\ x\in I,
\end{equation}
and moreover $|\partial_s\widetilde{\gamma}(s)|=1$ for a.e.\ $s\in J$ (cf.\ \cite[Theorem 3.2]{Hajlasz_2003}).
This in particular implies that $\mathcal{L}(\gamma)=\mathcal{L}(\widetilde{\gamma})$, i.e., the arclength reparameterization does not change the length.
We now verify the invariance of the area functional.
For the reparameterized curve we have the following chain rule \cite[Corollary 3.66]{Leoni_2017}:
\begin{equation}
\partial_x\gamma(x)=\partial_s\widetilde{\gamma}(\sigma(x))\partial_x\sigma(x) \qquad \text{for a.e.}\ x\in I,
\end{equation}
where the right-hand side is interpreted as zero whenever $\partial_x\sigma(x)=0$.
This together with (\ref{eq:A.1}) implies that
\begin{equation}\label{eq:A.3}
\int_I\gamma\cdot R\partial_x\gamma dx = \int_I [\widetilde{\gamma}(\sigma(x))\cdot R\partial_s\widetilde{\gamma}(\sigma(x))]\partial_x\sigma(x)dx.
\end{equation}
We also have the change of variables formula for the area functional \cite[Corollary~3.78]{Leoni_2017}:
\begin{equation*}
\int_I [\widetilde{\gamma}(\sigma(x))\cdot R\partial_s\widetilde{\gamma}(\sigma(x))]\partial_x\sigma(x)dx = \int_J \widetilde{\gamma}(s)\cdot R\partial_s\widetilde{\gamma}(s)ds,
\end{equation*}
which implies with (\ref{eq:A.3}) that $\mathcal{A}(\gamma)=\mathcal{A}(\widetilde{\gamma})$, i.e., the area functional is also invariant with respect to reparameterization.

\section{First variation and multiplier method}\label{section:B}

In this section we rigorously carry out first variation arguments for vector-valued Lipschitz functions.
We use a special case of the Lagrange multiplier method \cite[Proposition 1 in Section 4.14]{Zeidler_1995}:

\begin{theorem}\label{theorem:B.1}
Let $X$ be a real Banach space, $U\subset X$ be an open set, $f,g\in C^1(U;\mathbb{R})$, and $A_g:=\{u\in X \mid g(u)=0 \}$.
If $u\in U$ is a minimizer (or maximizer) of $f$ in $A_g$, then either $Dg(u)=0$, or there exists $\lambda\in\mathbb{R}$ such that $Df(u)+\lambda Dg(u)=0$.
\end{theorem}

For a given curve $\gamma \in W^{1,\infty}(I;\mathbb{R}^2)$ of positive constant speed,
i.e., $|\partial_x \gamma|\equiv\mathcal{L}(\gamma)>0$ a.e.\ in $I$,
we define $F(u):=\mathcal{L}(\gamma+u)$ and $G(u):=\mathcal{A}(\gamma+u)-1$ for $u\in W^{1,\infty}(I;\mathbb{R}^2)$.
Recall the well-known calculations that for any $u,\varphi\in W^{1,\infty}(I;\mathbb{R}^2)$,
\begin{align}\label{eq:B.1}
DG(u)(\varphi) &=-\frac{1}{2}\int_I [R\partial_x(\gamma+u)\cdot\varphi + R^{-1}(\gamma+u)\cdot\partial_x\varphi] dx,
\end{align}
and, provided that $\|\partial_x u\|_\infty<\mathcal{L}(\gamma)$,
\begin{align}\label{eq:B.2}
DF(u)(\varphi)=\int_I\frac{\partial_x(\gamma+u)}{|\partial_x(\gamma+u)|}\cdot\partial_x\varphi dx,
\end{align}
where the integrand is well defined since $|\partial_x(\gamma+u)|\geq\mathcal{L}(\gamma)-\|u\|_\infty>0$.
Notice that $DF$ and $DG$ are continuous with respect to $u$.
In addition,
\begin{align}
DG(0)(\varphi) &= -\int_I R\partial_x\gamma\cdot\varphi dx + \left[R^{-1}\gamma\cdot\varphi\right]_0^1, \label{eq:B.3}\\
DF(0)(\varphi) &= \int_I\partial_x\gamma\cdot\partial_x\varphi dx. \label{eq:B.4}
\end{align}

The purpose of this section is to prove the smoothness and properties of minimizers of $\mathcal{L}$ in the set $X_\theta'= \{ \gamma \in X_\theta \mid \mathcal{A}(\gamma)=1 \}$,
where $X_\theta$ was defined by~\eqref{eq:3.1}.

\begin{lemma}\label{theorem:B.2}
Let $\bar{\gamma}$ be a minimizer of $\mathcal{L}$ in $X_\theta'$, where $\theta\in(0,2\pi]$.
Then $\bar{\gamma}$ is smooth on $\bar{J}$ and moreover of constant curvature.
\end{lemma}

\begin{proof}
Without loss of generality, we may assume that $\bar{\gamma}$ is of constant speed, i.e., $|\partial_x\bar{\gamma}|\equiv\mathcal{L}(\bar{\gamma})>0$ a.e.\ in $I$.
We use Theorem \ref{theorem:B.1} for the spaces
\begin{equation*}
	X := W^{1,\infty}_0(I;\mathbb{R}^2), \quad U := \{\varphi\in X\mid \|\partial_x\varphi\|_\infty<\mathcal{L}(\bar{\gamma})\},
\end{equation*}
and the functionals
\begin{equation*}
	f(u) := \mathcal{L}(\bar{\gamma}+u), \quad g(u) := \mathcal{A}(\bar{\gamma}+u)-1.
\end{equation*}
Note that $Df$ and $Dg$ are continuous on $U$, cf.\ (\ref{eq:B.2}) and (\ref{eq:B.1}).
By the minimality of $\bar{\gamma}$, the zero function $u=0\in U$ is a minimizer of $f$ subject to $g=0$, and hence Theorem \ref{theorem:B.1} implies that, since $Dg(0)\neq0$,
there is $\lambda\in\mathbb{R}$ such that $Df(0)+\lambda Dg(0)=0$.
By \eqref{eq:B.3}, \eqref{eq:B.4}, and using the integration by parts, we obtain for any $\varphi\in X$,
\begin{align*}
0=Df(0)(\varphi)+\lambda Dg(0)(\varphi)=-\int_I(\partial_x^2\bar{\gamma}+\lambda R\partial_x\bar{\gamma})\cdot\varphi dx,
\end{align*}
where $\partial_x^2\bar{\gamma}$ is first understood in the distributional sense but then a bootstrap argument makes the classical sense; in particular, $\bar{\gamma}\in C^\infty(\bar{I};\mathbb{R}^2)$.
The assertion that $\bar{\gamma}$ is of constant curvature follows by the fundamental lemma of calculus of variations and the facts that $R\partial_x\bar{\gamma}=\mathcal{L}(\bar{\gamma})\nu$
and $\partial_x^2\bar{\gamma}=\mathcal{L}(\bar{\gamma})^2\kappa\nu$, which follow since $\bar{\gamma}$ is of constant speed.
\end{proof}

\begin{lemma}\label{theorem:B.3}
Suppose that $\theta\in(0,2\pi)$.
Let $\bar{\gamma}$ be a minimizer of $\mathcal{L}$ in $X_\theta'$ such that $|\bar{\gamma}(0)|=|\bar{\gamma}(1)|\neq0$.
Then $\partial_x\bar{\gamma}(0)\cdot v_0=\partial_x\bar{\gamma}(1)\cdot v_\theta$.
\end{lemma}

\begin{proof}
We may again assume that $\bar{\gamma}$ is of constant speed.
We also use Theorem~\ref{theorem:B.1} as in the above proof, just replacing $X$ and $U$ by
\begin{align*}
\widetilde{X} &:=\{\varphi\in W^{1,\infty}_0(I;\mathbb{R}^2) \mid \exists\alpha\in\mathbb{R},\ \varphi(0)=\alpha v_0,\ \varphi(1)=\alpha v_\theta \},\\
\widetilde{U} &:=\{\varphi\in X\mid \|\partial_x\varphi\|_\infty<\mathcal{L}(\bar{\gamma}),\ |\varphi(0)|<|\bar{\gamma}(0)| \}.
\end{align*}
Calculating the first variation, we find that there is $\tilde{\lambda}\in\mathbb{R}$ such that for any $\varphi\in\widetilde{X}$,
\begin{align*}
0 &=-\int_I(\partial_x^2\bar{\gamma}+\tilde{\lambda}R\partial_x\bar{\gamma})\cdot\varphi dx + \tilde{\lambda}\left[R^{-1}\bar{\gamma}\cdot\varphi\right]_0^1 + \left[\partial_x\bar{\gamma}\cdot\varphi\right]_0^1\\
   &=: T_1 + T_2 + T_3.
\end{align*}
By restricting perturbations $\varphi$ so that $\varphi(0)=\varphi(1)=0$,
the same argument as in Lemma \ref{theorem:B.2} implies that $\partial_x^2\bar{\gamma}+\tilde{\lambda}R\partial_x\bar{\gamma}=0$ a.e.\ in $I$, and hence $T_1=0$ holds even for any $\varphi\in\widetilde{X}$.
In addition, $T_2=0$ also holds for any $\varphi\in\widetilde{X}$ thanks to the boundary conditions of $\bar{\gamma}\in X_\theta'$ and $\varphi\in\widetilde{X}$.
Therefore, we deduce that $T_3=0$ needs to hold for any $\varphi\in\widetilde{X}$, which directly implies the assertion.
\end{proof}

\section{Rotational symmetry preserving property}

For the reader's convenience, we give a proof of the rotational symmetry preserving property based on uniqueness and geometric invariance,
although this kind of argument would be standard in geometric flows; see e.g.\ the remark just after Theorem 1.1 in \cite{EMS_1998}.

Throughout this section, we let $\gamma_0$ be a smooth initial curve in the same sense as Section \ref{section:4} so that \eqref{eq:CDF} is well posed,
and $\gamma:\mathbb{S}^1\times[0,T)\to\mathbb{R}^2$ denote a unique solution to \eqref{eq:CDF}.
Recall that it is enough for our purpose if $\gamma_0$ is of class $C^{2,\alpha}$, thanks to \cite[Theorem 1.1]{EMS_1998}.

\begin{lemma}[Invariance under rotation and change of variables]\label{theorem:C.1}
Let $S$ be a rotation matrix in $\mathbb{R}^2$, and $\xi:\mathbb{S}^1\to\mathbb{S}^1$ be an orientation-preserving diffeomorphism.
Then $S\circ\gamma\circ(\xi\times{\rm Id})$ is a unique solution to \eqref{eq:CDF} starting from $S\circ\gamma_0\circ\xi$.
\end{lemma}

\begin{proof}
Since the initial condition $S\circ\gamma\circ(\xi\times{\rm Id})(\cdot,0)=S\circ\gamma_0\circ\xi$ holds, by uniqueness,
it suffices to show that $S\circ\gamma\circ(\xi\times{\rm Id})$ satisfies \eqref{eq:CDF} on $\mathbb{S}^1\times(0,T)$.
Using the linearity of $S$ and \eqref{eq:CDF} for $\gamma$, we have
\begin{align*}
\partial_t[S\circ\gamma\circ(\xi\times{\rm Id})] &=S\circ\partial_t\gamma\circ(\xi\times{\rm Id})\\
 &= S\circ[(\partial_s^2\kappa)\nu]\circ(\xi\times{\rm Id}).
\end{align*}
Therefore, what we need to show is that for any fixed $t\in(0,T)$, the expression $S\circ[(\partial_s^2\kappa)\nu]\circ(\xi\times{\rm Id})$ is nothing but the right-hand side of (CDF)
for the curve $S\circ\gamma\circ(\xi\times{\rm Id})$; namely, letting $\gamma^*:=S\circ\gamma\circ(\xi\times{\rm Id})$, we show that
\begin{align*}
[(\partial_{s^*}^2\kappa^*)\nu^*](x,t)=S[(\partial_s^2\kappa^2)\nu](\xi(x),t),
\end{align*}
where $\partial_s$ (resp.\ $\partial_{s^*}$) denotes the arclength derivative, $\nu$ (resp.\ $\nu^*$) the unit normal,
and $\kappa$ (resp.\ $\kappa^*$) the curvature of the time-slice curve $\gamma(\cdot,t)$ (resp.\ $\gamma^*(\cdot,t)$).
This easily follows from the facts that $\nu^*(x,t)=S\nu(\xi(x),t)$ and that $\partial_{s^*}^2\kappa^*(x,t)=\partial_s^2\kappa(\xi(x),t)$,
which are just consequences of the definition $\gamma^*(x,t)=S\gamma(\xi(x),t)$ and standard geometric properties (e.g.\ invariance under change of variables).
\end{proof}

\begin{remark}
The same property holds for more general $S$ and $\xi$, namely, for any isometry $S$ and any diffeomorphism $\xi$, which may not preserve orientation.
\end{remark}

\begin{lemma}[Rotational symmetry preserving]\label{theorem:C.3}
If $\gamma_0$ is $(m,i)$-th rotationally symmetric, then so is $\gamma(\cdot,t)$ for all $t\in(0,T)$.
\end{lemma}

\begin{proof}
By assumption, letting $S:=R_{2\pi i/m}$ and $\xi(x):=x+1/m$, we have
\begin{align*}
\gamma_0\circ\xi=S\circ\gamma_0.
\end{align*}
Applying Lemma \ref{theorem:C.3} to the both sides, we find that the solution to \eqref{eq:CDF} starting from the left-hand side is $\gamma\circ(\xi\times{\rm Id})$,
while the right-hand side $S\circ\gamma$.
By uniqueness we have $\gamma\circ(\xi\times{\rm Id})=S\circ\gamma$, thus completing the proof.
\end{proof}

\end{appendix}

\end{document}